\documentclass[leqno]{amsproc}
\usepackage{times}
\usepackage{amsfonts,amssymb,amsmath,amsgen,amsthm}
\usepackage{hyperref}
\usepackage{color}
\usepackage{mathtools}
\newcommand{\msc}[2][2000]{%
  \let\@oldtitle\@title%
  \gdef\@title{\@oldtitle\footnotetext{#1 \emph{Mathematics subject
        classification.} #2}}%
}

\theoremstyle{plain}
\newtheorem{theorem}{Theorem} [section]

\newtheorem{lemma}[theorem]{Lemma}

\newtheorem{proposition}[theorem]{Proposition}

\theoremstyle{remark}
\newtheorem{remark}[theorem]{Remark}

\DeclarePairedDelimiter\abs{\lvert}{\rvert}%
\DeclarePairedDelimiter\norm{\lVert}{\rVert}%

\def\C{{\mathbb C}}
\def\R{{\mathbb R}}
\def\O{\mathcal O}
\def\F{\mathcal F}

\def\({\left(}
\def\){\right)}
\def\<{\left\langle}
\def\>{\right\rangle}
\def\le{\leqslant}
\def\ge{\geqslant}

\def\Eq#1#2{\mathop{\sim}\limits_{#1\rightarrow#2}}
\def\Tend#1#2{\mathop{\longrightarrow}\limits_{#1\rightarrow#2}}

\def\d{{\partial}}
\def\eps{\varepsilon}

\def\si{{\sigma}}

\DeclareMathOperator{\RE}{Re}
\DeclareMathOperator{\IM}{Im}

\numberwithin{equation}{section}

\begin{document}

\title[LogNLS with potential]{Logarithmic Schr\"odinger equation with
  quadratic potential} 
\author[R. Carles]{R\'emi Carles}
\address{Univ Rennes, CNRS\\ IRMAR - UMR 6625\\ F-35000
  Rennes, France}
\email{Remi.Carles@math.cnrs.fr}

\author[G. Ferriere]{Guillaume Ferriere}
\address{Institut de Recherche Math\'ematique Avanc\'ee\\ UMR 7501
Universit\'e de Strasbourg et CNRS\\
7 rue Ren\'e Descartes\\ 67000 Strasbourg\\ France}
\email{guillaume.ferriere@math.unistra.fr}
\begin{abstract}
  We analyze dynamical properties of the logarithmic Schr\"odinger
  equation under a quadratic potential. The sign of the
  nonlinearity is such that it is known that in the absence of external potential,
  every solution is dispersive, with a universal asymptotic
  profile. The introduction of a harmonic 
  potential generates solitary waves, corresponding to generalized
  Gaussons. We prove that they are orbitally stable,
  using an inequality related to relative entropy, which may be thought of as dual to
  the classical logarithmic Sobolev inequality. In the case of a
  partial confinement, we show a universal dispersive behavior for
  suitable marginals. For repulsive harmonic potentials, the
  dispersive rate is dictated by the potential, and no universal
  behavior must be expected.
\end{abstract}

\thanks{RC is supported by Rennes M\'etropole through its AIS program. }

\maketitle

\section{Introduction}
\label{sec:intro}

We consider the logarithmic Schr\"odinger equation in the presence of
an external potential,
\begin{equation}
  \label{eq:logNLSpot}
  i\d_t u +\frac{1}{2} \Delta u =V (x)u+ \lambda \ln\(|u|^2\)u  ,\quad u_{\mid t=0} =u_0  ,
\end{equation}
with $x\in \R^d$, $d\ge 1$, and $\lambda\in \R$. The potential $V$ is smooth and
real-valued, $V\in C^\infty(\R^d;\R)$, and the main results of the
present paper concern the case where $V$ is quadratic in $x$. This
equation is Hamiltonian; mass and energy are formally
independent of time:
\begin{equation}
  \label{eq:conserv}
 \begin{aligned}
  & M(u(t))=\|u(t)\|_{L^2(\R^d)}^2,\\
  & E(u(t)):=\frac{1}{2}\|\nabla u(t)\|_{L^2(\R^d)}^2+\int_{\R^d}
    V(x)|u(t,x)|^2dx\\
  &\phantom{ E(u(t)):=}+\lambda\int_{\R^d}
    |u(t,x)|^2\(\ln|u(t,x)|^2-1\)dx.
\end{aligned} 
\end{equation}
On the other hand, due to the presence of the
  potential $V$, the momentum, 
\[\IM \int_{\R^d} \bar u(t,x)\nabla u(t,x)dx,\] 
is not
  conserved. As it will not be considered in this paper, we do not
  discuss the evolution of this quantity.
The logarithmic nonlinearity was introduced in \cite{BiMy76} to
satisfy the following tensorization property: if $V$ decouples space
variables in the sense that
\begin{equation*}
  V(x) =\sum_{j=1}^d V_j(x_j),
\end{equation*}
where $V_j\in C^\infty(\R;\R)$, and if the initial datum is a tensor
product,
\begin{equation*}
  u_0(x) =\prod_{j=1}^d u_{0j}(x_j),
\end{equation*}
then the solution to \eqref{eq:logNLSpot} is given by
\begin{equation*}
  u(t,x) =\prod_{j=1}^d u_{j}(t,x_j), 
\end{equation*}
where each $u_j$ solves a one-dimensional equation,
\begin{equation*}
   i\d_t u_j +\frac{1}{2} \d_{x_j}^2 u_j =V_j (x_j)u_j+ \lambda
   \ln\(|u_j|^2\)u_j  ,\quad u_{j\mid t=0} =u_{0j} . 
 \end{equation*}
 Of course, a similar property remains if e.g. $V(x_1,x_2,x_3) =
 V_{12}(x_1,x_2)+V_3(x_3)$, and $u_0(x_1,x_2,x_3) =
 u_{012}(x_1,x_2)u_{03}(x_3)$. As can easily be checked, the above
 logarithmic nonlinearity is the only one satisfying this
 tensorization property.  This property will be evoked several times
 in the rest of this paper.
\smallbreak
 
The logarithmic nonlinearity has since been adopted in several
 physical models, in
quantum mechanics \cite{yasue},
quantum optics \cite{BiMy76,hansson,KEB00,buljan}, nuclear
physics \cite{Hef85}, Bohmian mechanics \cite{DMFGL03}, effective
quantum gravity \cite{Zlo10}, theory of 
superfluidity and  Bose-Einstein condensation \cite{BEC}. See also
\cite{ScottShertzer18,ShertzerScott20}. The papers
\cite{Zlo10,Zlo11,Bouharia2015} have provided evidences that the
logarithmic model may generalize the Gross-Pitaevskii equation, used in
the case of two-body interaction, to the case of three-body
interaction. 
\smallbreak

As proposed in \cite{Zlo10,Zlo11}, the logarithmic nonlinearity may
appear as a relevant model to provide a universal mechanism describing
the deformation of the quantum wave equation due to non-trivial
vacuum. It thus appears as a serious candidate to extend quantum
mechanics thanks to a nonlinear model, likely to help understand
quantum gravity. In the absence of external potential ($V=0$), the
case $\lambda<0$ is certainly the most physically relevant: stationary
solutions known as Gaussons are present and stable under the dynamics
(see below), while for $\lambda>0$, (enhanced) dispersion is always present; see
Theorem~\ref{theo:logNLSdisp} for a complete mathematical statement,
showing that it is indeed hopeless to look for stable structures if
$\lambda>0$. 
\smallbreak

In \cite{Bouharia2015}, the presence of an harmonic trap
was considered, in order to describe logarithmic BEC. As in the case
without potential, an important feature of the model is that many
Gaussian solutions are available, as explored into more details in
Section~\ref{sec:gaussian}. Stationary solutions (generalized Gaussons)
are still available in the case $\lambda<0$. Their stability was
investigated numerically in \cite{Bouharia2015}, and 
analyzed mathematically in \cite{ACS20}. The presence of the harmonic
trap implies that stable structures are now also present in the case
$\lambda>0$, as shown numerically in \cite{Bouharia2015}, and analyzed
more precisely in the present paper.

\smallbreak

Considering for $V$ a harmonic potential is not only physically
relevant, it  has also a natural mathematical motivation, which we now
explain.

\subsection{Mathematical background}
\label{sec:background}

The first mathematical study of \eqref{eq:logNLSpot} appears in
\cite{CaHa80} (see also \cite{CazCourant}), where the Cauchy problem is addressed for $V\in
L^p(\R^d)+L^\infty(\R^d)$, with $p\ge 1$ and $p>d/2$. As noticed in
\cite{BiMy79}, in the 
case $\lambda<0$ with $V=0$, \eqref{eq:logNLSpot} admits a solitary
wave of the form $e^{i\omega t }\phi(x)$, where $\phi$ is a Gaussian,
called \emph{Gausson}. The orbital stability of this solitary wave was
established in \cite{Caz83} in the radial case, and in \cite{Ar16} in the
general case.  Still in the case $\lambda<0$ with $V=0$, it was shown
in \cite{CaHa80} that if $u_0\in W$, the energy space given by
\begin{equation*}
  W=\left\{ f\in H^1(\R^d),\quad |f|^2\ln |f|^2\in L^1(\R^d)\right\},
\end{equation*}
then there exists a unique, global, solution $u\in
L^\infty(\R;W)$  (see also \cite{GuLoNi10,HayashiM2018}), and in
\cite{Caz83} that no such solution is 
dispersive. The global Cauchy problem in the case $V=0$, $\lambda>0$ was
treated in \cite{GuLoNi10} and \cite{CaGa18}.
\smallbreak

As noticed in \cite{BiMy76}, in the case $V=0$, $\lambda\in \R$, if $u_0$
is a Gaussian, then $u(t,\cdot)$ is a Gaussian for all time. The
evolution of Gaussians was analyzed more precisely in
\cite{BCST,CaGa18,FeDCDS}: in the case $\lambda<0$, they behave like
breathers, in the sense that $|u(t,\cdot)|$ is periodic in time, while
if $\lambda>0$, they all are dispersive with a similar rate
$\tau(t)$, and $\tau(t)^{d}
|u(t,x\tau(t))|^2$ has a universal limit, that is, a Gaussian whose
variance does not depend on the initial Gaussian $u_0$. 
\smallbreak

The latter property has been established for initial data which are
not necessarily Gaussian, in \cite{CaGa18}, and the description of the
convergence was refined in \cite{FeAPDE}. Denote
\begin{equation*}
  \Sigma =H^1\cap \F(H^1)= \left\{ f\in H^1(\R^d),\quad x\mapsto |x| f(x)\in L^2(\R^d)\right\}.
\end{equation*}
\begin{theorem}[\cite{CaGa18,FeAPDE}]\label{theo:logNLSdisp}
  Let $V=0$ and $\lambda>0$. For $u_0\in \Sigma\setminus\{0\}$, \eqref{eq:logNLSpot}
  has a unique solution $u\in L^\infty_{\rm
    loc}(\R;\Sigma))$. Introduce the solution $\tau\in C^\infty(\R)$ to the ODE
  \begin{equation}\label{eq:tau-libre}
  \ddot \tau = \frac{2\lambda }{\tau} \, ,\quad \tau(0)=1\, ,\quad \dot
  \tau(0)=0\, .
\end{equation}
Then, as
$t\to \infty$, 
$\tau(t)\sim 2t \sqrt{\lambda \ln t}$ and $ \dot
  \tau(t)\sim 2\sqrt{\lambda\ln  t} $. 
Introduce
$\gamma(x):=e^{-|x|^2/2}$, and 
 rescale the solution  to $v=v(t,y)$ by setting
\begin{equation}
  \label{eq:uvDMJ}
  u(t,x)
  =\frac{1}{\tau(t)^{d/2}}v\left(t,\frac{x}{\tau(t)}\right)
\frac{\|u_0\|_{L^2({\mathbb R}^d)}}{\|\gamma\|_{L^2({\mathbb R}^d)}} 
\exp \Big({i\frac{\dot\tau(t)}{\tau(t)}\frac{|x|^2}{2}} \Big) . 
\end{equation}
Then we have
\begin{equation*}
   \int_{{\mathbb R}^d}
  \begin{pmatrix}
    1\\
y\\
|y|^2
  \end{pmatrix}
|v(t,y)|^2dy\Tend t \infty 
 \int_{{\mathbb R}^d}
  \begin{pmatrix}
    1\\
y\\
|y|^2
  \end{pmatrix}
  \gamma^2(y)dy ,
\end{equation*}
and
\begin{equation*}
  |v(t,\cdot)|^2 \mathop{\rightharpoonup}\limits_{t\to \infty}
  \gamma^2 
\quad  \text{weakly in }L^1({\mathbb R}^d)  . 
\end{equation*}
Finally, denoting by  $  W_1$ the Wasserstein distance\footnote{For $\nu_1$
and $\nu_2$ probability measures, 
\begin{equation*}
  W_1(\nu_1,\nu_2)=\inf \left\{ \int_{{\mathbb R}^d\times
    {\mathbb R}^d}|x-y|d\mu(x,y);\quad (\pi_j)_\sharp \mu=\nu_j\right\},
\end{equation*}
where $\mu$ varies among all probability measures on ${\mathbb R}^d\times
{\mathbb R}^d$, and $\pi_j:{\mathbb R}^d\times {\mathbb R}^d\to {\mathbb R}^d$ denotes the canonical
projection onto the $j$-th factor. See e.g. \cite{Vi03}},
there exists $C$ such that
\begin{equation*}
 W_1\(\frac{|v(t)|^2}{\pi^{d/2}},\frac{\gamma^2}{\pi^{d/2}}\)\le
  \frac{C}{\sqrt{\ln t}},\quad t\ge e.
\end{equation*}
\end{theorem}
The case where $V$ is harmonic, $V(x) = |x|^2$, was considered in
\cite{ACS20}, in the case $\lambda<0$: there exists an analogue of the
Gausson, that is a solution of the form $b_0e^{i\omega t
}e^{-a_0|x|^2/2}$, and it is orbitally stable. In the following
statement, we adapt the numerical values to \eqref{eq:logNLSpot} and
unify the statements from \cite{Ar16} (without potential) and
\cite{ACS20} (with potential, where \eqref{eq:logNLSpot} is no longer 
translation invariant):
\begin{theorem}[\cite{Ar16,ACS20}]\label{theo:stabGaussonHarmo}
  Let $d\ge 1$ and $\lambda<0$. Suppose that
  \begin{equation*}
    V(x) = \frac{\kappa(\kappa+2\lambda)}{2}|x|^2,\quad \kappa\ge
    -2\lambda>0. 
  \end{equation*}
  Then the (generalized) Gausson is given by $\phi_\nu(x)
  =e^{-\frac{\nu+\kappa d/2}{2\lambda}} e^{-\kappa
    |x|^2/2}$, for $\nu\in \R$. It  generates a standing wave $u(t,x) = \phi_\nu(x)
  e^{i\nu t}$ solution to \eqref{eq:logNLSpot}, which is
  orbitally stable in the energy space:
  For any $\eps>0$, there exists $\eta>0$ such that if $u_0\in X$
satisfies $\|u_0-\phi_\nu\|_X<\eta$, then the solution $u$ to
\eqref{eq:logNLSpot} exists for all $t\in \R$, and
\begin{itemize}
\item \emph{Case without potential:} $\kappa=-2\lambda$ and $X=W$,
  \begin{equation*}
    \sup_{t\in \R}\inf_{\theta\in \R}\inf_{y\in \R^d}\|u(t)
    -e^{i\theta}\phi_\nu(\cdot-y)\|_W<\eps.
  \end{equation*}
 \item \emph{Case with potential:}  $\kappa>-2\lambda$ and $X=\Sigma$, 
 \begin{equation*}
    \sup_{t\in \R}\inf_{\theta\in \R}\|u(t)
    -e^{i\theta}\phi_\nu\|_\Sigma<\eps.
  \end{equation*}
 \end{itemize}
\end{theorem}

\subsection{Main results}
\label{sec:main}

To make things clear, we first show that the Cauchy problem
\eqref{eq:logNLSpot} is 
globally well-posed in $\Sigma$, provided that the potential $V$ is
smooth and at most quadratic. This naturally generalizes the results
known in the case of a power nonlinearity (see e.g. \cite{Ca11}). 
\begin{proposition}\label{prop:cauchy}
  Let $V\in C^\infty(\R^d;\R)$, at most quadratic, in the sense that
  $\d^\alpha V\in L^\infty(\R^d)$ as soon as $|\alpha|\ge 2$. For
  $u_0\in \Sigma$, there exists a unique solution $u\in L^\infty_{\rm
    loc}(\R;\Sigma)\cap C(\R;L^2(\R^d))$ to \eqref{eq:logNLSpot}. Moreover, the mass
  $M(u(t))$ and the energy $E(u(t))$ are
  independent of time.  
\end{proposition}
The above result was established in \cite{ACS20} in the case where $V$
is an isotropic harmonic potential, and $\lambda<0$. We prove
Proposition~\ref{prop:cauchy} by adapting the strategy from
\cite{CaGa18}, which is different from the one in \cite{ACS20},
inspired by \cite{CaHa80}. 
\begin{remark}
  The case of a time dependent potential $V$ could be considered as
  well, with very few modifications regarding
  Proposition~\ref{prop:cauchy} (see \cite{Ca11}), as well as in the
  description of the propagation of Gaussian functions in
  Section~\ref{sec:gaussian}. 
\end{remark}
In the rest of this introduction, and for the other results, we assume
$\lambda>0$.

\subsubsection{Full harmonic confinement}

We first consider the case of a full, isotropic confinement: $V(x) =
\tfrac{\omega^2}{2}|x|^2$. This confinement completely alters the
dynamics of the case $\lambda>0$, since solitary waves now exist,
while from Theorem~\ref{theo:logNLSdisp},  all $\Sigma$-solutions are
dispersive when $V=0$. 
 To make the connexion with the case $\lambda <0$ considered in
\cite{ACS20} explicit, rewrite $\omega^2$ as $\omega^2 =
\kappa(\kappa+2\lambda)$, $\kappa>0$. 

\begin{theorem}\label{theo:stab-orb}
    Let $d\ge 1$ and $\lambda>0$. Suppose that
  \begin{equation*}
    V(x) = \frac{\kappa(\kappa+2\lambda)}{2}|x|^2,\quad \kappa>0. 
  \end{equation*}
  Then the (generalized) Gausson is given by $\phi_\nu(x)
  =e^{-\frac{\nu+\kappa d/2}{2\lambda}} e^{-\kappa
    |x|^2/2}$, for $\nu\in \R$. It  generates a standing wave $u(t,x) = \phi_\nu(x)
  e^{i\nu t}$ solution to \eqref{eq:logNLSpot}, which is
  orbitally stable in the energy space:
  For any $\eps>0$, there exists $\eta>0$ such that if $u_0\in \Sigma$
satisfies $\|u_0-\phi_\nu\|_\Sigma<\eta$, then the solution $u$ to
\eqref{eq:logNLSpot} exists for all $t\in \R$, and
 \begin{equation*}
    \sup_{t\in \R}\inf_{\theta\in \R}\|u(t)
    -e^{i\theta}\phi_\nu\|_\Sigma<\eps.
  \end{equation*}
\end{theorem}
We see in particular that in the limit case $V=0$, corresponding to
$\kappa=0$, $\phi_\nu$ is no longer an $L^2$-function. To prove
Theorem~\ref{theo:stab-orb}, we essentially resume the approach from
\cite{ACS20}, based on the Cazenave-Lions method \cite{CaLi82}, as well as on
a variational characterization of the generalized Gausson. In the case
$\lambda<0$, this characterization relies on the logarithmic Sobolev
inequality and the description of equality cases. In the present
framework, the logarithmic Sobolev
inequality is replaced by a rather natural counterpart, involving a momentum instead
of a derivative, see Lemma~\ref{lem:logMoment}. 
\begin{remark}
 The case of anisotropic confinement in all
directions,
\begin{equation*}
  V(x) =\sum_{j=1}^d\frac{\omega_j^2}{2}x_j^2,\quad \omega_j>0,
\end{equation*}
 can be addressed with straightforward adaptations, by considering suitable anisotropic Gaussian functions in
all steps of the proof, including Lemma~\ref{lem:logMoment}.
\end{remark}

\subsubsection{Partial harmonic confinement}
 In the same spirit as e.g. \cite{AnCaSi15,BBJV17}, we now assume
 $d\ge 2$, and that $V$ is confining in some but not all directions:
 suppose that the space variable is now $(x',x'')\in \R^{d_1}\times
 \R^{d_2}$, with $d_1,d_2\ge 1$, $d_1+d_2=d$, and
 \begin{equation*}
   V(x',x'') = \frac{\omega^2}{2}|x'|^2.
 \end{equation*}
We show that dispersion is always present in $x''$, in the sense that
Theorem~\ref{theo:logNLSdisp} remains valid for $u$, provided that
suitable integration in $x'$ is considered.

\begin{theorem}\label{theo:logNLSpartial}
  Let $(x',x'')\in \R^{d_1}\times
 \R^{d_2}$, with $d_1,d_2\ge 1$, $d_1+d_2=d$, and
 \begin{equation*}
   V(x',x'') = \frac{\omega^2}{2}|x'|^2, \quad \omega>0.
 \end{equation*}
 Suppose $\lambda>0$. Let $u_0\in \Sigma\setminus\{0\}$, 
and resume the notation
$\gamma(x''):=e^{-|x''|^2/2}$.
Introduce
\begin{equation*}
  \rho(t,y) :=\tau(t)^{d_2}\int_{\R^{d_1}}\left|u\(t,x',y\tau(t)\)\right|^2dx'\times\frac{\pi^{d_2}}{\|u_0\|_{L^2(\R^d)}^2}.
\end{equation*}
Then we have
\begin{equation*}\label{eq:moments}
   \int_{{\mathbb R}^{d_2}}
  \begin{pmatrix}
    1\\
y\\
|y|^2
  \end{pmatrix}
\rho(t,y)dy\Tend t \infty 
 \int_{{\mathbb R}^{d_2}}
  \begin{pmatrix}
    1\\
y\\
|y|^2
  \end{pmatrix}
  \gamma^2(y)dy ,
\end{equation*}
and
\begin{equation*}\label{eq:weaklimitv}
  \rho(t,\cdot)\mathop{\rightharpoonup}\limits_{t\to \infty}
  \gamma^2 
\quad  \text{weakly in }L^1({\mathbb R}^d)  ,
\end{equation*}
along with
\begin{equation*}
 W_1\(\frac{|\rho(t)|^2}{\pi^{d_2/2}},\frac{\gamma^2}{\pi^{d_2/2}}\)\le
  \frac{C}{\sqrt{\ln t}},\quad t\ge e.
\end{equation*}
\end{theorem}
Theorem~\ref{theo:logNLSpartial} can be informally restated by saying
that
\begin{equation*}
  \int_{\R^{d_1}}|u(t,x',x'')|^2dx'
\end{equation*}
is dispersive in the variable $x''$, with rate $\tau(t)^{-d_2}$. Such
a dispersive property on a marginal was used in 
 \cite{AnCaSi15} as a first step (following from Morawetz estimates)
 to prove scattering for NLS with a 
 power-like nonlinearity and a partial confinement. In the present
 framework, no scattering result can be inferred.
\smallbreak 
 Indeed, considering tensorized initial data of the form $u_0(x) =
u_{01}(x')u_{02}(x'')$, we have $u(t,x)=u_1(t,x')u_2(t,x'')$, where
$u_1$ solves \eqref{eq:logNLSpot} with a fully confining potential,
and $u_2$ solves \eqref{eq:logNLSpot} with $V=0$, hence obeys
Theorem~\ref{theo:logNLSdisp} (with $d$ replaced by $d_2$). In
particular, $u_1$ may correspond to solutions described in
Theorem~\ref{theo:stab-orb}, or be a Gaussian breather (see
Section~\ref{sec:gaussian}), showing that integrating with respect to
$x'$ first, in the above result, makes perfect sense, and the dynamics
in the $x'$ variable is independent of the dispersion stated in
Theorem~\ref{theo:logNLSpartial}.

\subsubsection{Repulsive harmonic potential}
We finally present some results in the repulsive harmonic case,
\begin{equation}\label{eq:repulsive}
  i\d_t u +\frac{1}{2}\Delta u = -\omega^2\frac{|x|^2}{2}u +\lambda
  u\ln\(|u|^2\),\quad u_{\mid t=0}=u_0,
\end{equation}
for $x\in \R^d$ and $\omega,\lambda>0$.
In the case $\lambda=0$, any defocusing energy-subcritical power-like
nonlinearity 
 $|u|^{2\si}u$, $0<\si<\tfrac{2}{(d-2)_+}$,   is short range for
 scattering, since the repulsive harmonic 
   potential induces an exponential decay in time (\cite{CaSIMA}). In
   particular, in the dispersive frame (meaning after rescaling the
   wave function in terms of the time dependent dispersion, $e^{\omega
     t}$ in this case), any asymptotic profile can
   be reached. 
   On the other hand for $\omega=0$, Theorem~\ref{theo:logNLSdisp}
   shows that there is only one profile (if one considers the modulus
   only)  which can be reached in the dispersive frame. 
The case $\omega,\lambda>0$ is therefore a borderline case from these
two perspectives: do we have scattering, or a universal behavior? In
the next result, we give a partial answer to this question, showing
the same dispersive rate as in the linear case $\lambda=0$, and ruling
out a universal behavior in the sense of Theorem~\ref{theo:logNLSdisp}.
\begin{proposition}\label{prop:repulsive}
  Let $\omega,\lambda>0$ and let $\tau_-$ be the solution to the ODE
  \begin{equation}\label{eq:tau-moins}
     \ddot\tau_- = \omega^2\tau_-
     +\frac{2\lambda}{\tau_-},\quad
     \tau_-(0)=1,\quad 
     \dot\tau_-(0)=0. 
   \end{equation}
   There exists $\mu_\infty>0$ such that, as $t\to\infty$,
   $\tau_-(t)\sim \mu_\infty e^{\omega t}$, $\dot \tau_-(t)\sim
   \omega\mu_\infty e^{\omega t}$. For $u_0\in 
   \Sigma$ and the solution $u\in L^\infty_{\rm loc}(\R;\Sigma)$ to
   \eqref{eq:repulsive},  consider the rescaling
   \begin{equation*}
  u(t,x)
  =\frac{1}{\tau_-(t)^{d/2}}v\left(t,\frac{x}{\tau_-(t)}\right)
\exp \({i\frac{\dot\tau_-(t)}{\tau_-(t)}\frac{|x|^2}{2}} \) .
\end{equation*}
Then $v$ is bounded and not dispersive, in the sense that
\begin{equation*}
 \sup_{t\ge 0}\int_{{\mathbb R}^d}\left(1+|y|^2+\left|\ln
    |v(t,y)|^2\right|\right)|v(t,y)|^2dy <\infty.
\end{equation*}
We can find two initial data
 \begin{equation*}
    u_{01}(x)= e^{-|x|^2/2 + i\beta_1 |x|^2/2},\quad u_{02}(x)= e^{-|x|^2/2 +
      i\beta_2 |x|^2/2},\quad \beta_j>0,
  \end{equation*}
  such that the corresponding $v_j$'s satisfy
  \begin{equation*}
    |v_j(t,\cdot)|^2\Tend t \infty \gamma_j^2 \text{ in }L^1(\R^d),
    \quad \text{with }\gamma_1\not =\gamma_2. 
  \end{equation*}
\end{proposition}
In the same spirit as the comments following the statement of
Theorem~\ref{theo:logNLSpartial}, we emphasize that the dynamics
associated to the logarithmic nonlinearity \eqref{eq:logNLSpot} is
completely different from the dynamics for NLS with a power-type
nonlinearity. Typically,  in \cite{CaDCDS}, the case of  a
saddle potential, for $x\in \R^d$, $d\ge 2$,
\begin{equation*}
  V(x) =\omega_1^2 x_1^2-\omega_2^2 x_2^2,
\end{equation*}
was considered. There is confinement in the first direction, and a
strong dispersion in  the second one. 
In the
case of a defocusing power-like nonlinearity $|u|^{2\sigma}u$, we have global existence in
$\Sigma$. Given $u_0\in \Sigma$, at least if $\sigma\ge
2/d$ and $\omega_2$
is sufficiently 
large compared to $\omega_1$,  there is scattering in
the sense that there exist $u_\pm\in\Sigma$ such that
\begin{equation*}
  \|U_V(-t)u(t)-u_\pm\|_{\Sigma}\Tend t {\pm \infty}0,\quad U_V(t)=
  e^{-itH},\quad H=-\frac{1}{2}\Delta +V.
\end{equation*}
Proposition~\ref{prop:repulsive} shows that for a logarithmic
nonlinearity, there is still an exponential dispersion in $x_2$, but
the tensorization property implies that no scattering result can
hold. 
\subsection{Outline of the paper}
\label{sec:comments}

In Section~\ref{sec:cauchy}, we show how to prove
Proposition~\ref{prop:cauchy}. In Section~\ref{sec:algebra}, we
describe more precisely some remarkable properties related to the
logarithmic nonlinearity in the presence of a quadratic potential:
invariances, special transforms, and propagation of Gaussian
data. Theorem~\ref{theo:stab-orb} is proven in Section~\ref{sec:full},
relying on Lemma~\ref{lem:logMoment}, which can be thought of as a
dual to the logarithmic Sobolev inequality. In
Section~\ref{sec:partial}, we address the proof of
Theorem~\ref{theo:logNLSpartial}, by showing essentially how to
consider a suitable setting in order to relate
this result to Theorem~\ref{theo:logNLSdisp}. Finally,
Proposition~\ref{prop:repulsive} is proved in
Section~\ref{sec:repulsive}.

\section{Cauchy problem}
\label{sec:cauchy}

In this section, we briefly explain the proof of
Proposition~\ref{prop:cauchy}, which can be addressed like in the case
without potential considered in \cite{CaGa18}. The main issue is that
the logarithmic nonlinearity is not Lipschitz continuous at the
origin. Proceeding as in \cite{CaGa18}, we first regularize the
nonlinearity by saturating the logarithm near zero, and consider the
sequence of approximate solutions given by
\begin{equation}
  \label{eq:logNLSpot-app}
  i\d_t u^\eps +\frac{1}{2} \Delta u ^\eps =V (x)u^\eps+ \lambda \ln\(\eps+|u^\eps|^2\)u^\eps  ,\quad u^\eps_{\mid t=0} =u_0  ,
\end{equation}
where $V$, $\lambda$ and $u_0$ are as in
Proposition~\ref{prop:cauchy}, and $\eps>0$. For fixed $\eps>0$, the
above nonlinearity is locally Lipschitzian, with moderate growth at
infinity, ensuring that it is $L^2$-subcritical. Since the potential $V$ is
smooth and at most quadratic, local-in-time Strichartz estimates are
available for $H=-\frac{1}{2}\Delta+V$, the same as in the case $V=0$
(see e.g. \cite{Ca11} and references therein). Therefore, for every
$\eps>0$, there is a unique, global solution at the $L^2$
level. Again, at fixed $\eps>0$, the nonlinearity is smooth, so higher
regularity is propagated: $u^\eps,\nabla u^\eps, xu^\eps\in
C(\R;L^2(\R^d))$.

\smallbreak
The sequence $(u^\eps)_{0<\eps\le 1}$ converges, thanks to compactness
arguments based on uniform a priori estimates. Since $\lambda\in \R$
and $V$ is real-valued, the $L^2$-norm of $u^\eps$ is independent of
time, $\|u^\eps(t)\|_{L^2}=\|u_0\|_{L^2}$. For $1\le j\le d$,
differentiating \eqref{eq:logNLSpot-app} with respect to $x_j$ yields
\begin{align*}
  i{\partial}_t {\partial}_j u_\eps +\frac{1}{2} \Delta
  {\partial}_ju_\eps &=V(x)\d_j u^\eps + \d_j V(x)
 u^\eps+\lambda \ln\left(\eps+|u_\eps|^2\right){\partial}_ju_\eps
  \\
  &\quad + 2\lambda \frac{1}{\eps+|u_\eps|^2}  \RE (
    {\bar u}_\eps {\partial}_ju_\eps) u_\eps.
\end{align*}
By assumption on $V$, $|\d_j V(x)|\lesssim 1+|x|$, so the standard
 $L^2$ estimate yields
\begin{equation*}
  \frac{d}{dt}\|\nabla u^\eps(t)\|_{L^2}^2\le C \(
  \|u^\eps(t)\|_{L^2}^2+\|x u^\eps(t)\|_{L^2}^2 + \|\nabla
  u^\eps(t)\|_{L^2}^2\),
\end{equation*}
where $C$ is independent of $\eps $. Similarly, 
\begin{align*}
  i{\partial}_t\( x_j u_\eps \)+\frac{1}{2} \Delta
 (x_ju_\eps) &= \d_j u^\eps+V(x)x_j u^\eps  +\lambda
               \ln\left(\eps+|u_\eps|^2\right)x_ju_\eps, 
\end{align*}
hence
\begin{equation*}
  \frac{d}{dt}\|x u^\eps(t)\|_{L^2}^2\le 2 \int_{\R^d} |x u^\eps(t,x)|
  |\nabla u^\eps(t,x)|dx\le\|x u^\eps(t)\|_{L^2}^2 + \|\nabla
  u^\eps(t)\|_{L^2}^2. 
\end{equation*}
In view of the conservation of the mass, Gronwall lemma implies that
there exists $C$ independent of $\eps$ 
such that
\begin{equation*}
  \|x u^\eps(t)\|_{L^2}^2 + \|\nabla
  u^\eps(t)\|_{L^2}^2 \le C\(\|u_0\|_{L^2}^2+\|x u_0\|_{L^2}^2 + \|\nabla
  u_0\|_{L^2}^2 \)  e^{C|t|}, \quad t\in \R. 
\end{equation*}
Therefore, we have compactness in space for the sequence
$(u^\eps)_\eps$. Compactness in time follows from
\eqref{eq:logNLSpot-app}. Arzela-Ascoli theorem yields  a converging
subsequence, hence the existence part of 
Proposition~\ref{prop:cauchy}.
\smallbreak

The conservation of mass and energy can be proven
like in \cite{CaHa80} (see also \cite{CazCourant}).
\smallbreak

Uniqueness follows from the remark that any solution $u\in
L^\infty_{\rm loc}(\R;\Sigma)$ to \eqref{eq:logNLSpot} actually
belongs to $C(\R;L^2(\R^d))$, since the equation implies $\d_t u \in
L^\infty_{\rm loc}(\R;H^{-1}+\F(H^{-1}))$ (see \cite{CaGa18}), and from
the argument discovered in \cite{CaHa80}:
\begin{lemma}[Lemma 9.3.5 from \cite{CazCourant}]\label{lem:unique}
  We have 
  \begin{equation*}
    \left| \IM\left(\left(z_2\ln|z_2|^2
      -z_1\ln|z_1|^2\right)\left(\bar z_2-\bar z_1\right)\right)\right|\le
    4|z_2-z_1|^2 \, ,\quad \forall 
    z_1,z_2\in {\mathbb C} \, .
  \end{equation*}
\end{lemma}
For two solutions $u_1,u_2\in L^\infty_{\rm loc}(\R;\Sigma)$
to \eqref{eq:logNLSpot}, the difference $w=u_2-u_1$ solves
\[
 i{\partial}_t w +\frac{1}{2} \Delta w =V(x)w+ \lambda \(
 \ln\left(|u_2|^2\right)u_2- \ln\left(|u_1|^2\right)u_1\),\quad
 w_{\mid t=0}=0, 
\]
and the standard $L^2$ estimate yields, along with the above lemma,
 \begin{align*}
\frac12 \frac d {dt} \|w(t) \|_{L^2({\mathbb R}^d)}^2& = \lambda  \IM
\int_{{\mathbb R}^d} \( \ln\left(|u_2|^2\right)u_2-
\ln\left(|u_1|^2\right)u_1\) (\bar u_2-\bar u_1) (t) \, dx 
\\
& \le 4 \lambda \|w(t) \|_{L^2({\mathbb R}^d)}^2 ,
 \end{align*}
 hence $w\equiv 0$ from Gronwall lemma.

\section{Remarkable algebraic properties of \eqref{eq:logNLSpot}}
\label{sec:algebra}

\subsection{Invariants}

Apart from the conservations of mass and energy,
\eqref{eq:logNLSpot} is invariant with respect to translation in
time, but not in space in general, due to the potential $V$. When $V$
is zero or exactly quadratic, there is a Galilean invariance (whose
expression depends on the signature of $\nabla^2V$). Having the
tensorization property in mind, recall the formulas for $d=1$:
\begin{itemize}
\item Case $V=0$: if $u(t,x)$ solve \eqref{eq:logNLSpot}, then  for
  any $v\in \R$, so does
  $u(t,x-vt)e^{ivx-iv^2 t/2}$.
\item Case $V(x)=\omega^2x^2/2$:  if $u(t,x)$ solve
  \eqref{eq:logNLSpot}, then  for any $v\in \R$, so does 
  \[u\(t,x-v\frac{\sin(\omega t)}{\omega}\)e^{ivx\cos(\omega t)-iv^2
    \cos(\omega t)\frac{\sin(\omega t)}{2\omega}}.\]
\item Case $V(x)=-\omega^2x^2/2$:  if $u(t,x)$ solve
  \eqref{eq:logNLSpot}, then  for any $v\in \R$, so does 
   \[u\(t,x-v\frac{\sinh(\omega t)}{\omega}\)e^{ivx\cosh(\omega t)-iv^2
    \cosh(\omega t)\frac{\sinh(\omega t)}{2\omega}}.\]
\end{itemize}
The logarithmic nonlinearity causes a rather unique invariance, as
noticed in \cite{CaGa18} for $V=0$, a property which remains in the
presence of $V$: If $u$ solves \eqref{eq:logNLSpot}, then for all
$k\in \C$, so does
\begin{equation*}
  k u(t,x) e^{-it\lambda\ln|k|^2}.  
\end{equation*}
This shows that the size of the initial data alters the dynamics only
through a purely time dependent oscillation, a feature which is
fairly unusual for a nonlinear equation. In addition, considering
$k>0$, differentiating the above formula with respect to $k$ and
letting $k\to 0$ shows that for $t>0$ arbitrarily small, the flow map
$u_0\mapsto u(t)$ cannot be $C^1$, whichever function spaces are
considered for $u_0$ and $u(t)$, respectively; it is at most
Lipschitzian. 
\subsection{Special potentials}

It is a standard fact that if $V$ is linear in $x$, $V(x)= E\cdot x$
for some fixed  $E\in \R^d$, the influence of $V$ is explicit, in the
sense that if $u$ solves \eqref{eq:logNLSpot} with this $V$, and $v$
solves \eqref{eq:logNLSpot} with $V=0$ (and the same initial datum),
then $u$ and $v$ are related through
\begin{equation*}
  v(t,x)= u\left(t, x -\frac{t^2}{2}E \right)e^{i\left( 
tE\cdot x -\frac{t^3}{3}|E|^2\right)}.
\end{equation*}
See e.g. \cite{Ca11}, where this formula is extended to the case where
$E$ depends on time.

On the other hand, lens transforms (see e.g. \cite{Ca11}) seem to be
useless in the case of a 
logarithmic nonlinearity. If $u$ solves \eqref{eq:logNLSpot} with
$V=0$, for $\omega>0$ and $|t|<\pi/(2\omega)$, set
\begin{equation}\label{eq:uu+}
w^+(t,x)= \frac{1}{\left(\cos (\omega t)
\right)^{d/2}}e^{-i\frac{\omega}{2}|x|^2 \tan (\omega
t)} u\left(\frac{\tan (\omega t)}{\omega}, \frac{x}{\cos (\omega
t)} \right)\, .
\end{equation}
Then $w^+$ solves
\begin{equation*}
  i\d_t w^+ +\frac{1}{2} \Delta w^+ =\frac{\lambda}{\cos(\omega t)^2} \ln\(\cos(\omega t)^d|w^+|^2\)w^+ +\omega^2\frac{|x|^2}{2}w^+,\quad w^+_{\mid t=0} =u_0 \, .
\end{equation*}
The time dependent factor in front of the nonlinearity shows a strong
difference with \eqref{eq:logNLSpot}, and the equation in $w^+$ is not
necessarily more pleasant to study. The case of a repulsive harmonic
potential $V(x)=-\omega^2|x|^2/2$ is similar (replace
$\omega$ by $i\omega$). 
\subsection{Propagation of Gaussian initial data}
\label{sec:gaussian}
A remarkable feature of  \eqref{eq:logNLSpot} is that when the
potential is a polynomial of degree at most two in space, then an
initial Gaussian data evolves as a Gaussian for all time. Plugging a
time-dependent Gaussian function into the logarithmic nonlinearity, we
readily see that this remarkable property is related to the well-known
fact that the same holds in the case of the \emph{linear} Schr\"odinger
equation with potential; see e.g. \cite{Hag80,Hag81,Hel75,Hepp}. 

In view of the tensorization property described in the introduction,
we consider the case $d=1$, and a quadratic potential,
\begin{equation}
  \label{eq:logNLSquad}
  i\d_t u +\frac{1}{2} \d_x^2 u =\lambda \ln\(|u|^2\)u +\Omega\frac{x^2}{2}u ,\quad u_{\mid t=0} =u_0  ,
\end{equation}
where $\Omega\in \R$ is a constant.  We seek $u(t,x) =
b(t)e^{-a(t)x^2/2}$ (in particular $u_0$ is Gaussian). Plugging this
into \eqref{eq:logNLSquad}, we 
find:
\begin{equation*}
  i\dot b = \frac{1}{2}ab+\lambda b\ln|b|^2\quad ;\quad i\dot a
  =a^2+2\lambda\RE a-\Omega. 
\end{equation*}
Seeking $a$ under the form
\begin{equation*}
  a=\frac{1}{\tau^2} -i\frac{\dot \tau}{\tau}
\end{equation*}
leads to
\begin{equation}\label{eq:tau-general}
  \ddot \tau = \frac{2\lambda}{\tau}+\frac{1}{\tau^3}-\Omega \tau. 
\end{equation}
As long as the solution is smooth, multiply by $\dot \tau$ and
integrate in time: 
\begin{equation}\label{eq:taupoint-general}
  (\dot \tau)^2 = C_0 +4\lambda \ln\tau -\frac{1}{\tau^2} -\Omega
  \tau^2. 
\end{equation}
In the case of a (constant) harmonic potential, we write
$\Omega=\omega^2$: $\tau$ is bounded,
regardless of the sign of $\lambda$. Typically, we have generalized
Gaussons for each sign of $\lambda$: $u(t,x)
  = C_\nu e^{-k x^2/2+i\nu t}$ solves \eqref{eq:logNLSquad} if and only
  if
  \begin{equation*}
    \omega^2 -2\lambda k=k^2\quad;\quad -\nu -
    \frac{k}{2}=\lambda \ln(C_\nu^2). 
  \end{equation*}
For $u$ to be an $L^2$ function in space, only one choice is possible
for $k$,
\begin{equation*}
  k = -\lambda +\sqrt{\lambda^2+\omega^2}. 
\end{equation*}
Like in \cite{FeDCDS}, we see that $\tau$ is always periodic, and
Gaussian data propagate as breathers in general, as a solitary
wave in the specific case discussed above. Writing
$\omega^2=\kappa(\kappa+2\lambda)$, $\kappa>\max (0,-2\lambda)$, we
recover the generalized Gausson  of the introduction,
$\phi_\nu(x)=e^{-\frac{\nu+\kappa  /2}{2\lambda}} e^{-\kappa
  |x|^2/2} $, with $k=\kappa$ (and $d=1$ here).
\smallbreak

In the case of a (constant) repulsive harmonic potential, we write
$\Omega=-\omega^2$, and \eqref{eq:tau-general} becomes
\begin{equation}
  \label{eq:tau0}
  \ddot\tau = \omega^2\tau +\frac{2\lambda}{\tau}+\frac{1}{\tau^3}.
\end{equation}
Assuming $\tau(0)>0$, we see that $\tau$ remains positive and bounded
away from zero for all
time, since \eqref{eq:taupoint-general} becomes
\begin{equation}\label{eq:derivee}
  \dot\tau(t)^2= C_0 +\omega^2\tau(t)^2 +4\lambda \ln \tau(t)-\frac{1}{\tau(t)^2}.
\end{equation}
Therefore, in the case $\lambda>0$, there exists $\delta>0$ such that
\begin{equation*}
  \tau(t)\ge \delta,\quad \forall t\in \R. 
\end{equation*}
We infer trivially
\begin{equation*}
   \ddot\tau \ge \omega^2 \delta,
 \end{equation*}
hence
 \begin{equation}\label{eq:DV}
  \tau(t)\Tend t \infty \infty,\quad \dot\tau(t)\Tend t \infty \infty. 
\end{equation}
Therefore, we may approximate $\tau$ by the solution to
\begin{equation}
  \label{eq:tau-exp}
  \ddot\tau = \omega^2\tau .
\end{equation}
In particular, the presence of the repulsive harmonic potential
causes dispersion with an exponential dispersive rate. We analyze more
precisely this situation in Section~\ref{sec:repulsive}.

\section{Full harmonic confinement}
\label{sec:full}

In this section, we prove Theorem~\ref{theo:stab-orb}, by following
the strategy of Cazenave and Lions \cite{CaLi82}, and proving some
rigidity property of the generalized Gausson. We recall that
\begin{equation*}
     V(x) = \frac{\kappa(\kappa+2\lambda)}{2}|x|^2,\quad \kappa>0,
\end{equation*}
and we denote $\omega^2 = \kappa(\kappa+2\lambda)$ for conciseness.
\subsection{Technical preliminary}

 We
introduce the action and the Nehari functional:
\begin{align*}
    S_\nu (u) &\coloneqq E(u) + \nu \|u\|_{L^2}^2, \\
    I_\nu (u) &\coloneqq \norm{\nabla u}_{L^2}^2 + \omega^2 \norm{x
                u}_{L^2}^2 + 2 \lambda \int_{\R^d} \abs{u}^2
                \ln{\abs{u}^2} d x + 2 \nu \norm{u}_{L^2}^2=
                2S_\nu(u)+2\lambda\|u\|_{L^2}^2 .
\end{align*}
We also define the quantity
\begin{align*}
    D (\nu) &= \inf \{ S_{\nu} (u) \, | \, u \in \Sigma (\R^d)
              \setminus \{ 0 \}, I_{\nu} (u) = 0 \} \\ 
        &= - \lambda \sup \{ \norm{u}_{L^2}^2 \, | \, u \in \Sigma
          (\R^d) \setminus \{ 0 \}, I_{\nu} (u) = 0 \}, 
\end{align*}
and the set of ground states by
\begin{equation*}
    \mathcal{G}_\nu \coloneqq \{ \phi \in \Sigma (\R^d) \setminus \{ 0 \} \, | \, I_{\nu} (u) = 0, S_\nu (\phi) = D(\nu) \}.
\end{equation*}
At this stage, we emphasize a major difference
  between the case of a power nonlinearity and the logarithmic
  nonlinearity. For power-type nonlinearities, it is standard to
 either  minimize the action, or minimize the energy with a fixed
 mass, the two approaches being equivalent in the case of homogeneous
 nonlinearities. In the case of a logarithmic nonlinearity, the Nehari 
 functional is constant along the flow of the equation, and minimizing
the action without constraint, or the energy with a fixed mass, does not
seem adequate. This characterization of the ground state in the case
of a logarithmic nonlinearity is already present in \cite{Ar16,ACS20}.
A key step of the analysis consists in showing that $\mathcal{G}_\nu
=\{ e^{i\theta} \phi_\nu,\ \theta\in \R\}$, with $\phi_\nu$ defined in
Theorem~\ref{theo:stab-orb},
\begin{equation*}
  \phi_\nu(x) = e^{-\frac{\nu+\kappa d/2}{2\lambda}} e^{-\kappa
    |x|^2/2}.
\end{equation*}

First, we note that
the energy functional $E$ defined in \eqref{eq:conserv} is of class
$C^1$, and for $u \in \Sigma$, its Fr\'echet
derivative is given by 
    \begin{equation*}
        E'(u) = - \Delta u +  \omega^2 \abs{x}^2 u + \lambda u \ln{\abs{u}^2} .
    \end{equation*}
As a consequence,
    $S_\nu$ and $I_\nu$ are of class $C^1$, and for $u \in \Sigma$
    \begin{equation*}
        \langle S_\nu' (u), u \rangle =  I_\nu (u).
    \end{equation*}
We will also need the following compactness result:
\begin{lemma} \label{lem:compact_Sigma}
    For any sequence $(u_m)$ uniformly bounded in $\Sigma $, there exists a subsequence (still denoted $(u_m)$) and some $u \in \Sigma$ such that:
    \begin{itemize}
        \item $u_m \Tend m \infty u$ in $L^2 (\R^d)$,
        \item $u_m \Tend m \infty u$ a.e. in $\R^d$,
        \item The following convergence also holds:
        \begin{equation*}
            \lim_{m \to \infty} \int_{\R^d} \abs{u_m}^2 \ln{\abs{u_m}^2} dx = \int_{\R^d} \abs{u}^2 \ln{\abs{u}^2} dx.
        \end{equation*}
    \end{itemize}
\end{lemma}

\begin{proof}[Sketch of the proof]
    The first two points are standard. For the third point, using the
    fact that for all $\delta>0$, there exists $C_\delta> 0$ such that for all $y \ge 0$ 
    \begin{equation*}
        \abs{y^2 \ln{y^2}} \le C_\delta (y^{2 - \delta} + y^{2 + \delta}),
    \end{equation*}
   we get,
    \begin{equation*}
        \left|\int_{\R^d} \abs{u_m - u}^2 \ln{\abs{u_m - u}^2} dx\right| \le
        C \( 1 + \norm{u_m}_{\Sigma}^2 + \norm{u}_{\Sigma}^2 \)
        \norm{u_m - u}_{L^2} \Tend m \infty 0.
    \end{equation*}
  The third point then follows from
 Br\'ezis-Lieb lemma \cite{BrezisLieb}, in the convex case (see also
 \cite[Lemma~2.3]{Ar16}). 
\end{proof}

\subsection{Variational analysis}

The main novelty to prove orbital stability in the case $\lambda>0$ is 
 a result which may be viewed as
a dual of the celebrated logarithmic Sobolev inequality:
\begin{lemma}\label{lem:logMoment}
  Let $f\in \F(H^1(\R^d))$ and $a>0$:
  \begin{equation}\label{eq:logSobDual}
    -
    \int_{\R^d}|f(x)|^2 \ln\(\frac{|f(x)|^2}{\|f\|_{L^2}^2}\)dx
    \le a\int_{\R^d} |x|^2|f(x)|^2dx+\frac{d}{2}\|f\|_{L^2}^2
    \ln\frac{\pi}{a}. 
\end{equation}
There is equality if and only if $|f(x)|= ce^{-a|x|^2/2}$, with $c=
\|f\|_{L^2}(a/\pi)^{d/2}$. \\ 
Assuming $\|f\|_{L^2}=1$, we also have
 \begin{equation*}
   -\int_{\R^d}|f(x)|^2 \ln\(|f(x)|^2\)dx\le \frac{d}{2}\ln\(\frac{2
    e\pi }{d}\int_{\R^d} |x|^2|f(x)|^2dx\). 
\end{equation*}
\end{lemma}
\begin{remark}
 The above inequalities are not translation invariant, just like the usual
 logarithmic Sobolev inequality is not translation invariant on the
 Fourier side.  However, we must take much more than translation on
 the Fourier side into account, as the inequality is obviously
 invariant by multiplication by $e^{i\phi(x)}$ for $\phi$ real-valued,
 not necessarily linear in $x$. 
\end{remark}
The second inequality is the counterpart of the classical optimal
logarithmic Sobolev 
inequality, for $\|f\|_{L^2}=1$,
 \begin{equation*}
  \int_{\R^d}|f(x)|^2 \ln\(|f(x)|^2\)dx\le \frac{d}{2}\ln\(\frac{d
  }{2e\pi}\|\nabla f\|_{L^2}^2\) , 
\end{equation*}
established in \cite{Weissler78}.
\begin{proof}
  For $a>0$, consider the normalized Gaussian
  \begin{equation*}
    \nu_a(x) = \(\frac{a}{\pi}\)^{d/2}e^{-a|x|^2}.
  \end{equation*}
  It is normalized to ensure that $\nu_a$ is a probability
  density. In view of the Csisz\'ar-Kullback inequality  (see
e.g. \cite[Th.~8.2.7]{LogSob}), for $\mu,\nu$ probability densities,
\begin{equation*}
  \|\mu-\nu\|_{L^1({\mathbb R}^d)}^2\le 2 \int_{\R^d}
  \mu(x)\ln \left(\frac{\mu(x)}{\nu(x)}\right) dx. 
\end{equation*}
 Let $\mu(x)=|f(x)|^2/\|f\|_{L^2}^2$ and $\nu=\nu_a$: since the above
 relative entropy  (right-hand side) is non-negative, we have
 \begin{equation*}
   \int_{\R^d}|f(x)|^2\ln\(\frac{|f(x)|^2}{\|f\|_{L^2}^2\nu_a(x)}\)dx\ge 0.
 \end{equation*}
 Replacing $\nu_a$ by its explicit value and expanding  yields \eqref{eq:logSobDual}.
 \smallbreak

 Moreover, equality holds in \eqref{eq:logSobDual} if and only if the
 relative entropy is zero,
and Csisz\'ar-Kullback inequality implies that $\mu(x) = \nu_a(x)$. 
\smallbreak
 
 Suppose $\|f\|_{L^2}=1$. The map $a\mapsto \alpha a -\beta \ln a
 +c$, $\alpha, 
 \beta>0$ reaches its 
 minimum for $a= \beta/\alpha$, so we optimize \eqref{eq:logSobDual}
 with $\alpha = \|xf\|_{L^2}^2$, $\beta=\frac{d}{2}$ and
 $c= \frac{d}{2}\ln \pi$,
 \begin{equation*}
   -\int_{\R^d}|f(x)|^2 \ln\(|f(x)|^2\)dx\le \frac{d}{2}\ln\(\frac{2
     e}{d}\|xf\|_{L^2}^2\) +\frac{d}{2}\ln \pi , 
 \end{equation*}
 hence the announced inequality.
\end{proof}

\begin{lemma}[Ground energy and ground states] \label{lem:ground}
    Let $\nu \in \R$. Then the quantity $D(\nu)$ is given by
    \begin{equation*}
        D (\nu) = - \lambda \norm{\phi_\nu}_{L^2}^2 = - \lambda \pi^\frac{d}{2} \kappa^{- \frac{d}{2}} e^{-\frac{\nu+\kappa d/2}{\lambda}}.
    \end{equation*}
    Moreover,
    \begin{equation*}
        \mathcal{G}_\nu = \left\{ e^{i \theta} \phi_\nu, \theta \in [0, 2 \pi[ \right\}.
    \end{equation*}
\end{lemma}

\begin{proof}
    Let $u \in \Sigma \setminus \{ 0 \}$ such that $I_{\nu} (u) = 0$.
    First, we apply Lemma~\ref{lem:logMoment} with $a=
    \kappa$:
  \begin{equation} \label{eq:appl_logSobDual}
        2 \lambda \kappa \norm{x u}_{L^2}^2 + 2 \lambda \int_{\R^d} |u(x)|^2 \ln{\abs{u (x)}^2} dx \ge 2 \lambda \norm{u}_{L^2}^2 \( \ln{\norm{u}_{L^2}^2} - \frac{d}{2} \ln\frac{\pi}{\kappa} \).
    \end{equation}
    Using the fact that $\omega^2 = \kappa^2 + 2 \lambda \kappa$, we thus get
    \begin{equation} \label{eq:1st_est_Sigma}
        0 = I_\nu (u) \ge \norm{\nabla u}_{L^2}^2 + \kappa^2 \norm{x u}_{L^2}^2 + 2 \lambda \norm{u}_{L^2}^2 \Bigl( \ln{\norm{u}_{L^2}^2} - \frac{d}{2} \ln\frac{\pi}{\kappa} \Bigr) + 2 \nu \norm{u}_{L^2}^2.
    \end{equation}
    Moreover, since the ground state of the harmonic oscillator
    $-\Delta+\kappa^2|x|^2$ is $e^{-\kappa|x|^2/2}$, the bottom of the
    spectrum is $\kappa d$, thus
    \begin{equation} \label{eq:quantum_harm_energy}
        \norm{\nabla u}_{L^2}^2 + \kappa^2 \norm{x u}_{L^2}^2 \ge \kappa d \norm{u}_{L^2}^2,
    \end{equation}
    with equality if and only if there exists $\mu \in \mathbb{C}$ such that $u (x) = \mu \, e^{- \kappa \abs{x}^2/2}$.
    Thus, there holds
    \begin{align*}
        0 \ge 2 \lambda \norm{u}_{L^2}^2 \( \ln{\norm{u}_{L^2}^2} - \frac{d}{2} \ln\frac{\pi}{\kappa} + \frac{\kappa d}{2 \lambda} + \frac{\nu}{\lambda} \).
    \end{align*}
    Therefore, since $u \neq 0$, we get:
    \begin{equation*}
        \ln{\norm{u}_{L^2}^2} - \frac{d}{2} \ln\frac{\pi}{\kappa} + \frac{\kappa d}{2 \lambda} + \frac{\nu}{\lambda} \le 0,
    \end{equation*}
    which yields
    \begin{equation} \label{eq:L2_ineq}
        \norm{u}_{L^2}^2 \le \pi^\frac{d}{2} \kappa^{- \frac{d}{2}} e^{-\frac{\nu}{\lambda} - \frac{\kappa d}{2 \lambda}}.
    \end{equation}
    Since this inequality holds for all $u \in \Sigma\setminus \{ 0
    \}$ such that $I_{\nu} (u) = 0$, we get
    \[D(\nu) \ge - \lambda
    \pi^\frac{d}{2} \kappa^{- \frac{d}{2}} e^{ -\frac{\nu}{\lambda} - \frac{\kappa d}{2
        \lambda}}.\]
  Moreover, there is equality in this last
    inequality if and only if there is equality in all the previous
    inequalities \eqref{eq:appl_logSobDual},
    \eqref{eq:quantum_harm_energy} and \eqref{eq:L2_ineq}. 
    From the previous discussion, equality in both
    \eqref{eq:appl_logSobDual} and \eqref{eq:quantum_harm_energy} is
    equivalent to $u (x) = \mu \, e^{- \kappa\abs{x}^2/2}$ for some
    $\mu \in \mathbb{C}$. 
    Putting this condition in the last case of equality shows that
    $\abs{\mu}^2 = e^{-\frac{\nu+\kappa d/2}{\lambda}}$. 
    Therefore, all the inequalities are equalities if and only if there exists $\theta \in [0, 2 \pi[$ such that $u = e^{i \theta}\phi_\nu$. The conclusion then readily follows.
\end{proof}

\begin{remark}
  If we had to consider equality in \eqref{eq:appl_logSobDual} only, then $\theta$
  would be a \emph{map} from $\R^d$ to $\R$, not necessarily
  constant. On the other hand, equality in
  \eqref{eq:quantum_harm_energy} forces $\theta$ to be constant. 
\end{remark}
\subsection{Minimizing sequences}

As a final preparation for the proof of Theorem~\ref{theo:stab-orb},
we characterize minimizing sequences.

\begin{lemma} \label{lem:bounded_set}
    $\{ u \in \Sigma \setminus \{ 0 \}, I_{\nu} (u) = 0 \}$ is a
    bounded subset of $\Sigma $. 
\end{lemma}

\begin{proof}
    Let $u \in \Sigma \setminus \{ 0 \}$ such that $I_{\nu} (u) = 0$.
    By definition of $D(\nu)$, we know that
    \begin{equation*}
        \norm{u}_{L^2}^2 \le - \lambda^{-1} D(\nu).
    \end{equation*}
    Moreover, using \eqref{eq:1st_est_Sigma}, we get
    \begin{equation*}
        \norm{\nabla u}_{L^2}^2 + \kappa^2 \norm{x u}_{L^2}^2 \le - 2 \lambda \norm{u}_{L^2}^2 \ln{\norm{u}_{L^2}^2} + 2 \lambda \norm{u}_{L^2}^2 \Bigl( \frac{d}{2} \ln\frac{\pi}{\kappa} + 2 \nu \Bigr).
    \end{equation*}
    Using the fact that $-x \ln{x} \le e^{-1}$ for all $x \ge 0$, we get
    \begin{equation*}
        \norm{\nabla u}_{L^2}^2 + \kappa^2 \norm{x u}_{L^2}^2 \le 2 \lambda e^{-1} - 2 D(\nu) \( \frac{d}{2} \ln\frac{\pi}{\kappa} + 2 \nu \),
    \end{equation*}
  hence the conclusion, since the $L^2$-norm is controlled by $-D(\nu)$.
\end{proof}

\begin{lemma} \label{lem:min_seq}
    Let $(u_n)$ be a minimizing sequence for $D(\nu)$. Then there exists $\phi \in \mathcal{G}_\nu$ such that, up to a subsequence,
    \begin{equation*}
       \|u_n-\phi\|_\Sigma\Tend n \infty 0.
    \end{equation*}
\end{lemma}

\begin{proof}
    By definition of $D(\nu)$, we know that
    \begin{equation*}
        \norm{u_n}_{L^2}^2 \underset{n \to \infty}{\longrightarrow} - \lambda^{-1} D(\nu).
    \end{equation*}
    Moreover, Lemmas \ref{lem:bounded_set} and \ref{lem:compact_Sigma} show that there exists $\phi \in \Sigma$ such that, up to a subsequence,
    \begin{gather*}
        u_n \rightharpoonup \phi \qquad \textnormal{weakly in } \Sigma , \\
        u_n \to \phi \qquad \textnormal{strongly in } L^2 (\mathbb{R}^d), \\
        \lim_{n \to \infty} \int_{\mathbb{R}^d} \abs{u_n}^2 \ln{\abs{u_n}^2} dx = \int_{\mathbb{R}^d} \abs{\phi}^2 \ln{\abs{\phi}^2} dx.
    \end{gather*}
    In particular, this yields
    \begin{equation*}
        \norm{\phi}_{L^2} = - \lambda^{-1} D(\nu).
    \end{equation*}
    Moreover, using again \eqref{eq:1st_est_Sigma}, there holds
    \begin{equation*}
        \norm{\nabla u_n}_{L^2}^2 + \kappa^2 \norm{x u_n}_{L^2}^2 \le - 2 \lambda \norm{u_n}_{L^2}^2 \ln{\norm{u_n}_{L^2}^2} + 2 \lambda \norm{u_n}_{L^2}^2 \( \frac{d}{2} \ln\frac{\pi}{\kappa} -\frac{\nu}{\lambda} \),
    \end{equation*}
    and
    \begin{align*}
        - 2 \lambda \norm{u_n}_{L^2}^2 \ln{\norm{u_n}_{L^2}^2} &+ 2
    \lambda \norm{u_n}_{L^2}^2 \(\frac{d}{2}  \ln\frac{\pi}{\kappa}
      - \frac{\nu}{\lambda} \)\\
  \Tend n \infty     & 2 D(\nu) \( \ln{\(- \lambda^{-1} D(\nu)\)} - \frac{d}{2} \ln\frac{\pi}{\kappa} +\frac{\nu}{\lambda} \) = - \frac{\kappa d}{\lambda} D(\nu),
    \end{align*}
    by using Lemma \ref{lem:ground}.
    Thus, using again \eqref{eq:quantum_harm_energy}, we get
    \begin{align*}
        - \frac{\kappa d}{\lambda}  D(\nu) = \kappa d \norm{\phi}_{L^2} &\le \norm{\nabla \phi}_{L^2}^2 + \kappa^2 \norm{x \phi}_{L^2}^2 \\
            &\le \liminf \( \norm{\nabla u_n}_{L^2}^2 + \kappa^2 \norm{x u_n}_{L^2}^2 \) \\
            &\le \limsup \( \norm{\nabla u_n}_{L^2}^2 + \kappa^2 \norm{x u_n}_{L^2}^2 \) \le - \frac{\kappa d}{\lambda}  D(\nu).
    \end{align*}
    Therefore, there is equality in all of the above inequalities. In particular, 
    \begin{equation*}
        \norm{\nabla \phi}_{L^2}^2 + \kappa^2 \norm{x \phi}_{L^2}^2 = \lim \( \norm{\nabla u_n}_{L^2}^2 + \kappa^2 \norm{x u_n}_{L^2}^2 \),
    \end{equation*}
    which proves that 
    \begin{equation*}
        u_n \to \phi \qquad \textnormal{strongly in } \Sigma .
    \end{equation*}
    With all the previous convergences, the convergence of $I_\nu (u_n)$ to $I_\nu (\phi)$ is now obvious, and therefore $I_\nu (\phi) = 0$. Since $D(\nu) = - \lambda \norm{\phi}_{L^2}$, we obtain $\phi \in \mathcal{G}_\nu$.
\end{proof}

\subsection{Orbital stability}

Theorem~\ref{theo:stab-orb} is then classically proved by
contradiction. Assume that there exist $\eps > 0$, $u_{0n}\in\Sigma$
and $t_n\in\R$ such that
    \begin{gather}
        \inf_{\theta\in \R} \norm{u_{0n} - e^{i \theta}\phi_\nu}_{\Sigma}
        \Tend n\infty  0, \label{eq:conv_ground_state} \\
        \inf_{\theta\in\R} \norm{u_n (t_n) - e^{i \theta}\phi_\nu}_{\Sigma} > \eps \qquad \textnormal{for any } n \in \mathbb{N}. \label{eq:noconv_ground_state}
    \end{gather}
    Set $v_n = u_n (t_n)$. By \eqref{eq:conv_ground_state} and the
    conservation laws, we obtain
    \begin{gather*}
        \norm{v_n}_{L^2}^2 = \norm{u_{0n}}_{L^2}^2 \underset{n \to \infty}{\longrightarrow} \norm{\phi_\nu}_{L^2}^2, \\
        E(v_n) = E(u_{0n}) \underset{n \to \infty}{\longrightarrow} E(\phi_\nu).
    \end{gather*}
    In particular, there also holds
    \begin{equation*}
        S_\nu (v_n) \Tend n\infty S_\nu (\phi_\nu) = D(\nu).
    \end{equation*}
    Since $I_\nu (f) = 2 E(f) +2( \lambda+ \nu) \norm{u}_{L^2}^2$,
    \begin{equation*}
        I_\nu (v_n) \Tend n\infty I_\nu (\phi_\nu) = 0.
    \end{equation*}
    Next, define the sequence $f_n = \rho_n v_n$ with
    \begin{equation*}
        \rho_n = \exp{\( - \frac{I_\nu (v_n)}{2 \norm{v_n}_{L^2}^2} \)}.
    \end{equation*}
    It is clear that $I_\nu (f_n) = 0$ and $\lim_n \rho_n = 1$, so that $\norm{v_n - f_n}_{\Sigma} \to 0$ and thus $S_\nu (f_n) \to D(\nu)$. Therefore, $(f_n)$ is a minimizing sequence for $D(\nu)$. 
    Applying Lemma \ref{lem:min_seq}, there exists $\phi \in \mathcal{G}_\nu$ such that, up to a subsequence, $f_n \to \phi$ strongly in $\Sigma $. Therefore
    \begin{equation} \label{eq:conv_v_n}
        v_n \to \phi \qquad \textnormal{strongly in } \Sigma .
    \end{equation}
    Moreover, since $\phi \in \mathcal{G}_\nu$, Lemma \ref{lem:ground} gives the existence of $\theta \in [0, 2 \pi[$ such that $\phi = e^{i \theta} \phi_\nu$. Therefore, \eqref{eq:conv_v_n} contradicts \eqref{eq:noconv_ground_state}.

\section{Partial harmonic confinement}
\label{sec:partial}

\subsection{Preparation of the proof}

In this section, we explain how to prove Theorem~\ref{theo:logNLSpartial},
by adapting elements of proof introduced in \cite{CaGa18,FeAPDE}. Like
in Theorem~\ref{theo:logNLSdisp}, we rescale the initial unknown $u$ to
a new unknown $v$ through the formula
\begin{equation}
  \label{eq:uv-partial}
   u(t,x',x'')
  =\frac{1}{\tau(t)^{d_2/2}}v\left(t,x'\frac{x''}{\tau(t)}\right)
\frac{\|u_0\|_{L^2({\mathbb R}^d)}}{\|\gamma\|_{L^2({\mathbb R}^{d_2})}} 
\exp \Big({i\frac{\dot\tau(t)}{\tau(t)}\frac{|x''|^2}{2}} \Big) , 
\end{equation}
where we emphasize that at this stage, $\gamma$ is a function of the
$d_2$-dimensional variable $x''$. The function $\rho$ involved in
Theorem~\ref{theo:logNLSpartial} is then simply
\begin{equation*}
  \rho(t,y) =\int_{\R^{d_1}}|v(t,x',y)|^2dx',\quad y\in \R^{d_2}.
  \end{equation*}
  We readily check that \eqref{eq:logNLSpot} is then equivalent to
  \begin{equation*}
    i\d_t v +\frac{1}{2} \Delta_{x'} v +\frac{1}{2\tau(t)^2}\Delta_y v
    = \lambda \ln\(|v|^2\)v + \(\frac{\omega^2}{2}|x'|^2
    +\lambda|y|^2\)v+ \theta(t)v ,\quad v_{\mid t=0} =\alpha u_0  ,
  \end{equation*}
 with $\alpha =  \|\gamma\|_{L^2({\mathbb
       R}^{d_2})}/\|u_0\|_{L^2({\mathbb R}^d)}$, and where $\theta $ is real-valued and depends on time only,
  \begin{equation*}
    \theta(t) = -\lambda \( d_2 \ln \tau(t) -2 \ln \frac{\|u_0\|_{L^2({\mathbb R}^d)}}{\|\gamma\|_{L^2({\mathbb R}^{d_2})}} \)=-\lambda \( d_2 \ln \tau(t) -2 \ln \frac{\|u_0\|_{L^2({\mathbb R}^d)}}{\pi^{d_2/2}} \).
  \end{equation*}
Changing $v(t,x',y)$ to $v(t,x',y)e^{i\int_0^t \theta(s)ds}$ does not
affect the function $\rho$ involved in the statement of
Theorem~\ref{theo:logNLSpartial}, and removes the last term from the
above equation, so we may assume that $v$ solves
\begin{equation}\label{eq:v-partial}
   i\d_t v +\frac{1}{2} \Delta_{x'} v +\frac{1}{2\tau(t)^2}\Delta_y v
    = \lambda \ln\(|v|^2\)v + \(\frac{\omega^2}{2}|x'|^2
    +\lambda|y|^2\)v,\quad v_{\mid t=0} =v_0  ,
\end{equation}
where $v_0$ is such that
$\|v_0\|_{L^2(\R^d)}=\|\gamma\|_{L^2(\R^{d_2})}$. The equation
\eqref{eq:v-partial} is still Hamiltonian, but since it is no longer
autonomous, we have a dissipated energy: setting
\begin{align*}
  \mathcal E(t)&= \frac{1}{2}\norm{\nabla_{x'} v}_{L^2(\R^d)}^2 +
\frac{1}{2\tau(t)^2}\norm{\nabla_{y}
      v}_{L^2(\R^d)}^2+  \lambda \int_{\R^d}
  |v(t,x',y)|^2\ln |v(t,x',y)|^2dx'dy\\
  &\quad + \int_{\R^d}
 \(\frac{\omega^2}{2}|x'|^2
    +\lambda|y|^2\) |v(t,x',y)|^2dx'dy ,
\end{align*}
we have
\begin{equation*}
  \frac{d\mathcal E}{dt} = -\frac{\dot \tau}{\tau^3} \norm{\nabla_{y}
      v}_{L^2(\R^d)}^2\le 0.
\end{equation*}
This implies crucial a priori estimates:
\begin{lemma}\label{lem:apriori-v-partial}
  Let $u_0\in \Sigma$, $\lambda,\omega>0$. There exists $C>0$ such that
  \begin{align*}
   & \sup_{t\ge 0} \int_{\R^d} \( 1+ |x'|^2+|y|^2+\left|\ln
    |v(t,x',y)|^2\right|\)|v(t,x',y)|^2dx'dy\le C,\\
     &\sup_{t\ge 0} \(\frac{1}{2}\norm{\nabla_{x'} v}_{L^2(\R^d)}^2+ \frac{1}{2\tau(t)^2}
    \norm{\nabla_{y} v}_{L^2(\R^d)}^2 \)\le C.
  \end{align*}
  In addition,
  \begin{equation*}
   \int_0^\infty  \frac{\dot \tau(t)}{\tau(t)^3}
  \|\nabla_y v(t)\|_{L^2(\R^d)}^2<\infty.
  \end{equation*}
\end{lemma}
\begin{proof}[Sketch of the proof]
 We resume the main lines of the proof of
 \cite[Lemma~4.1]{CaGa18}. Write $\mathcal E=\mathcal E_+-\mathcal
 E_-$, where 
  \begin{align*}
    \mathcal E_+ &= \frac{1}{2}\norm{\nabla_{x'} v}_{L^2}^2 +
  \frac{1}{2\tau(t)^2}\norm{\nabla_{y}
      v}_{L^2}^2+\int_{\R^d}
 \(\frac{\omega^2}{2}|x'|^2
    +\lambda|y|^2\) |v(t,x',y)|^2dx'dy\\
  &\quad   + \lambda\int_{|v(t,x',y)|^2\ge 1}
                   |v(t,x',y)|^2\ln |v(t,x',y)|^2dx'dy   .
\end{align*}
Note that $\mathcal E_+$ is the sum of non-negative terms only.
Since $\mathcal E$ is non-increasing, for $t\ge 0$,
\begin{align*}
  \mathcal E_+(t) &\le \mathcal E(0)+\mathcal E_-(t)=  \mathcal E(0)+
  \lambda \int_{|v|\le 1}|v(t,x',y)|^2\ln \frac{1}{|v(t,x',y)|^2}dx'dy\\
  &\lesssim 1+ \int_{\R^d} |v(t,x',y)|^{2-\eps}dx'dy,
\end{align*}
for any $\eps>0$ sufficiently small. Then $\|v\|_{L^{2-\eps}}\lesssim
\|v\|_{L^2}^{1-\delta(\eps)}\|(|x'|+|y|)v\|_{L^2}^{\delta(\eps)}$ with
$\delta(\eps)\to 0 $ as $\eps\to 0$. In
view of the conservation of the $L^2$-norm of $v$ (which is equal to
the $L^2$-norm of $u$), we infer, for $0<\eps\ll 1$ ,
$\mathcal E_-\lesssim \mathcal E_+^{1/2}$ (for instance), hence
\begin{equation*}
  \mathcal E_+(t) \lesssim 1,\quad \forall t\ge 0. 
\end{equation*}
Therefore, each term in $\mathcal E_+$ is bounded, and so is $\mathcal
E_-$.

Since $\mathcal E$ is bounded, the integral of $\dot {\mathcal E}$ is
bounded, hence the last part of the lemma. 
\end{proof}

\subsection{Convergence of momenta}

The proof of \eqref{eq:moments} stems from the same arguments as in
\cite{CaGa18}. 

\subsubsection{Center of mass}
Introduce
\begin{equation*}
    I_{1} (t) := \IM \int_{\mathbb{R}^d} \bar v (t,x',y) \, \nabla_y
    v(t,x',y) dx' dy, \qquad I_{2}(t) := \int_{\mathbb{R}^d} y \, |v
    (t,x',y)|^2 dx' dy.
\end{equation*}
We compute
\begin{gather*}
    \dot{I}_{1} = -2 \lambda I_{2}, \qquad \dot{I}_{2}=
    \frac{1}{\tau^2 (t)} I_{1}, 
  \end{gather*}
  and so $\tilde I_2:=\tau I_2$ satisfies $ \ddot{\tilde{I}}_{2}= 0$.
  Since $\tau(t) \sim 2t\sqrt{\lambda \ln t}$ as $t\to \infty$, we
  infer that $ I_2(t) = \O (\ln t)^{-1/2}$.
  
\subsubsection{Second order momentum}
Rewriting in terms of $v$ the conservation of the energy for $u$,
\eqref{eq:conserv}, we find:
\begin{align*}
  E(u_0) &=\frac{1}{2}\norm{\nabla_{x'} v}_{L^2(\R^d)}^2 +
\frac{1}{2\tau(t)^2}\norm{\nabla_{y}
      v}_{L^2(\R^d)}^2+\frac{(\dot \tau)^2}{2}\int_{\R^d} |y|^2
 |v(t,x',y)|^2dx'dy\\
&\quad+ \frac{\dot \tau}{\tau}\IM\int_{\R^d} \bar v(t,x',y)y\cdot \nabla_yv(t,x',y)dx'dy
+ \frac{\omega^2}{2}\int_{\R^d}|x'|^2  |v(t,x',y)|^2dx'dy\\
&\quad+   \lambda \int_{\R^d}|v(t,x',y)|^2\ln |v(t,x',y)|^2dx'dy
 -\lambda d_2 \pi^{d_2}\ln \tau\\
  &\quad+2\lambda\|u_0\|_{L^2(\R^d)}^2\ln
           \(\frac{\|u_0\|_{L^2(\R^d)}}{\pi^{d_2/2}}\) .
\end{align*}
In view of Lemma~\ref{lem:apriori-v-partial}, the first, second, fifth,
sixth, and last terms on the right-hand side are bounded in time. In
view of Lemma~\ref{lem:apriori-v-partial} and Cauchy-Schwarz
inequality, the fourth term is $\O(\dot \tau)=\O(\sqrt{\ln t})$, and
we obtain
\begin{equation*}
  \frac{(\dot \tau)^2}{2}\int_{\R^d} |y|^2
 |v(t,x',y)|^2dx'dy-\lambda d_2 \pi^{d_2}\ln \tau=\O(\sqrt{\ln t}),
\end{equation*}
which yields, since $(\dot\tau)^2= 4\lambda\ln \tau$ (multiply
\eqref{eq:tau-libre} by $\dot \tau$ and integrate), 
\begin{equation*}
  \int_{\R^d} |y|^2|v(t,x',y)|^2dx'dy =  \frac{d_2}{2} \pi^{d_2} + \O\(\frac{1}{\sqrt{\ln
  t}}\) = \int_{\R^{d_2}} |y|^2\gamma(y)^2dy + \O\(\frac{1}{\sqrt{\ln
  t}}\) ,
\end{equation*}
hence \eqref{eq:moments}.
\subsection{Convergence of the profile}
 To prove the end of Theorem~\ref{theo:logNLSpartial}, we use a Madelung
 transform: define $R$, $J_1$  and $J_2$ by
 \begin{align*}
  &R(t,x',y)=|v(t,x',y)|^2, \quad J_1(t,x',y) = \IM \(\bar
    v(t,x',y)\nabla_{x'} v(t,x',y)\), \\
   &J_2(t,x',y) = \IM \(\bar
   v(t,x',y)\nabla_y v(t,x',y) \).
 \end{align*}
 We check that \eqref{eq:v-partial} implies
 \begin{align*}
   &\d_t R + \nabla_{x'}\cdot  J_1 +\frac{1}{\tau(t)^2}\nabla_y\cdot
     J_2=0,\\
   & \d_t J_2+\lambda \nabla_y R+2\lambda yR =
     \frac{1}{4\tau(t)^2}\Delta_y\nabla_y R
     -\frac{1}{\tau(t)^2}\nabla_y\cdot\(\RE\(\nabla_y v\otimes \nabla_y
     \bar v\)\)\\
   &\phantom{ \d_t J_2+\lambda \nabla_y R+2\lambda yR
     =}+\frac{1}{2}\nabla_{x'}\cdot \RE \(
    \bar v \(\(\nabla_{x'}\otimes\nabla_y\)v\)- \nabla_{x'}\bar
     v\otimes\nabla_y v\). 
 \end{align*}
We do not write the evolution law for $J_1$, as it is not needed for
the argument. 
\begin{remark}
    Despite the fact that $v (t)$ might not be $H^2$, the term $\bar v
    \(\nabla_{x'}\otimes\nabla_y\)v$ is still well defined in
    $\mathcal{D}' ((0, \infty) \times \mathbb{R}^d)$  owing to the fact
    that $v$ is in $L^\infty_{\rm loc}( (0, \infty); H^1)$ and to the relation:
    \begin{equation*}
    \bar v
    \(\nabla_{x'}\otimes\nabla_y\)v    = \nabla_y \(\bar v
    \nabla_{x'} v\) - \nabla_y\bar v\otimes\nabla_{x'}v.
    \end{equation*}
\end{remark}
We note that
\begin{equation*}
 \rho(t,y)=\int_{\R^{d_1}}R(t,x',y)dx', 
\end{equation*}
and we set
\begin{equation*}
  j(t,y) = \int_{\R^{d_1}}J_2(t,x',y)dx'.
\end{equation*}
These new unknowns solve, in $\mathcal D'((0,\infty)\times \R^{d_2})$:
\[
  \left\{
  \begin{aligned}
  &\d_t \rho +\frac{1}{\tau(t)^2}\nabla_y\cdot j =0,\\
  & \d_t j + \lambda \nabla_y \rho +2\lambda y \rho =  \frac{1}{4\tau(t)^2}\Delta_y\nabla_y \rho
     -\frac{1}{\tau(t)^2}\nabla_y\cdot \mu,
   \end{aligned}
   \right.
   \]
where
\begin{equation*}
  \mu(t,y) = \int_{\R^{d_1}}\RE\(\nabla_y v\otimes \nabla_y
     \bar v\)(t,x',y)dx'.
\end{equation*}
At this stage, in view of the a priori estimates provided by
Lemma~\ref{lem:apriori-v-partial}, we have the same ingredients are
those needed in  \cite[Section~5.3]{CaGa18} and
\cite[Lemma~4.4]{FeAPDE}, hence the weak convergence and the estimate
of the Wasserstein distance in Theorem~\ref{theo:logNLSpartial}.

\section{Repulsive harmonic potential}
\label{sec:repulsive}

In this section, we consider \eqref{eq:repulsive}, that is
\begin{equation*}
  i\d_t u +\frac{1}{2}\Delta u = -\omega^2\frac{|x|^2}{2}u +\lambda
  u\ln\(|u|^2\),\quad u_{\mid t=0}=u_0,
\end{equation*}
for $x\in \R^d$ and $\omega,\lambda>0$.

\subsection{Analysis of the dispersion}
\label{sec:dispersion}

We first resume the analysis started in Section~\ref{sec:gaussian} in
the Gaussian case,
and consider $\tau$ solving \eqref{eq:tau0}. We discuss the
dependence upon initial data at the end.

\subsubsection{Direct error estimate}

Let $\tau_{\rm eff}^T$ solve
\begin{equation*}
  \ddot\tau_{\rm eff}^T = \omega^2\tau_{\rm eff}^T,\quad \tau_{\rm
    eff}^T(T)=\tau(T),\ \dot\tau_{\rm    eff}^T(T)=\dot\tau(T).
\end{equation*}
Consider the error $e_T = \tau- \tau_{\rm eff}^T $. It solves
\begin{equation*}
  \ddot e_T-\omega^2 e_T = \frac{2\lambda}{\tau}+\frac{1}{\tau^3},
\end{equation*}
and Duhamel's formula reads
\begin{equation*}
  e_T(t) = 2\lambda\int_T^t \frac{\sinh\(\omega(t-s)\)}{\omega}
  \frac{ds}{\tau(s)} + \int_T^t \frac{\sinh\(\omega(t-s)\)}{\omega}
  \frac{ds}{\tau(s)^3}. 
\end{equation*}
In view of \eqref{eq:DV}, for $t\ge T\gg 1$,
\begin{equation*}
  |e_T(t)|\le \eps \int_T^t \frac{\sinh\(\omega(t-s)\)}{\omega}ds \le
  \frac{\eps}{\omega^2} \cosh\(\omega(t-T)\). 
\end{equation*}
On the other hand,
\begin{equation*}
\tau_{\rm eff}^T (t)\Eq t\infty \frac{1}{2}\(  \tau(T)+\frac{\dot
  \tau(T)}{\omega}\) e^{\omega(t-T)}. 
\end{equation*}
We infer from \eqref{eq:DV} that $\tau$ grows exponentially in time,
since $\tau(T)+\frac{\dot
  \tau(T)}{\omega}\gg \eps$. In view of \eqref{eq:derivee}, $\dot
\tau$ also grows exponentially in time:
\begin{equation}\label{eq:tau-exponential}
e^{\omega t}\gtrsim   \tau(t)\gtrsim e^{\omega t},\quad e^{\omega
  t}\gtrsim  \dot\tau(t)\gtrsim e^{\omega 
    t},\quad t>0.
\end{equation}
\subsubsection{Analysis of the main ODE}
Setting $\mu(t) = \tau(t)e^{-\omega t}$, we prove that
$\mu(t)\to \mu_\infty$ as $t\to \infty$. It solves
\begin{equation}\label{eq:gamma}
  \ddot \mu+2\omega \dot \mu =
  \frac{2\lambda}{\mu}e^{-2\omega t} +
  \frac{1}{\mu^3}e^{-4\omega t},
\end{equation}
hence
\begin{equation*}
  \frac{d}{dt}\(\dot \mu e^{2\omega t} \)= \frac{2\lambda}{\mu}+
  \frac{1}{\mu^3}e^{-2\omega t}>0.
\end{equation*}
We know from the previous section that
\begin{equation}\label{eq:gamma-borne}
  1\lesssim \mu(t)\lesssim 1,
\end{equation}
and the previous bound becomes
\begin{equation*}
   \frac{d}{dt}\(\dot \mu e^{2\omega t} \)\gtrsim 1.
\end{equation*}
The map $t\mapsto \dot \mu(t) e^{2\omega t}$ is increasing, and 
\begin{equation*}
  \dot \mu(t) e^{2\omega t}\gtrsim t-c\Tend t \infty \infty.
\end{equation*}
Then \eqref{eq:gamma} implies
\begin{equation*}
  \ddot \mu(t) e^{2\omega t}\lesssim -t +c\Tend t \infty -\infty.
\end{equation*}
Therefore, for $t$ sufficiently large, $\dot \mu$ is positive
decreasing, hence has a non-negative limit as $t\to\infty$, and $\ddot
\mu \in L^1$. But
\eqref{eq:gamma} implies that $\dot \mu$ and $\ddot \mu$ are
simultaneously $L^1$, so $\dot \mu\in L^1$ and the limit of $\dot \mu$ has to be zero. This yields the existence of
$\mu_\infty>0$ such that
\begin{equation*}
  \mu(t) \Tend t \infty \mu_\infty,\quad \text{hence}\quad
  \tau(t)\Eq t \infty \mu_\infty e^{\omega t}.
\end{equation*}
In view of \eqref{eq:derivee}, and since $\dot \tau>0$ (at least for
$t\gg 1$), we infer
\begin{equation*}
 \dot \tau(t)\Eq t \infty \omega\mu_\infty e^{\omega t}.
\end{equation*}

\subsubsection{Dependence of $\mu_\infty$}

First, we get another integral expression for $\tau$. Define $F (t):= \frac{2 \lambda}{\tau (t)} + \frac{1}{\tau (t)^3}$. Then there holds
\begin{equation*}
    \ddot{\tau} - \omega^2 \tau = F,
\end{equation*}
which leads to
\begin{equation*}
    \frac{d}{d t} \Bigl( e^{\omega t} (\dot{\tau} - \omega \tau) \Bigr) = F(t) \, e^{\omega t},
\end{equation*}
and thus
\begin{equation*}
    e^{\omega t} (\dot{\tau} - \omega \tau) = (\dot{\tau} (0) - \omega \tau (0)) + \int_0^t F(s) e^{\omega s} d s.
\end{equation*}
Therefore, we get
\begin{equation*}
    \tau (t) = e^{\omega t} \Bigl[ \tau (0) + \frac{1 - e^{-2 \omega t}}{2 \omega} (\dot{\tau} (0) - \omega \tau (0)) + \int_0^t e^{-2 \omega s} \int_0^s F(r) e^{\omega r} d r d s \Bigr].
\end{equation*}
After some easier computations, we get
\begin{multline} \label{eq:integral_tau}
    \tau (t) = \tau (0) \cosh (\omega t) + \frac{\dot{\tau} (0)}{\omega} \sinh (\omega t) + \frac{e^{\omega t}}{2 \omega} \int_0^t e^{- \omega r} F(r) d r \\ - \frac{e^{- \omega t}}{2 \omega} \int_0^t F(r) e^{\omega r} d r.
\end{multline}
We also already know that $\tau (t) \sim \mu_\infty e^{\omega t}$, so in particular $F (t) \sim \frac{2 \lambda}{\mu_\infty} e^{- \omega t}$, and thus
\begin{equation*}
    \int_0^t F(r) e^{\omega r} d r = \O (t), \qquad
    \int_0^\infty F(r) e^{- \omega r} d r < \infty.
\end{equation*}
Therefore, \eqref{eq:integral_tau} leads to an expression for $\mu_\infty$:
\begin{equation} \label{eq:integral_gamma}
    \mu_\infty = \tau (0) + \frac{\dot{\tau} (0)}{\omega} + \frac{1}{2 \omega} \int_0^\infty e^{- \omega r} \( \frac{2 \lambda}{\tau (r)} + \frac{1}{\tau (r)^3} \) d r.
\end{equation}
Moreover, from the fact that $F \ge 0$ and $e^{\omega t} e^{- \omega r} - e^{- \omega t} e^{\omega r} \ge 0$ for all $0 \le r \le t$, \eqref{eq:integral_tau} also gives for all $t \ge 0$:
\begin{equation*}
    \tau (t) \ge \tau (0) \cosh (\omega t) + \frac{\dot{\tau} (0)}{\omega} \sinh (\omega t).
\end{equation*}
Assume $\dot{\tau} (0) = \tau_1 \ge 0$ and $\tau (0) = \tau_0 > 0$. Then, the last term in the right-hand side of \eqref{eq:integral_gamma} can be estimated. First:
\begin{align*}
    0 \le \int_0^\infty e^{- \omega r} \frac{2 \lambda}{\tau (r)} d r
        &\le \int_0^\infty e^{- \omega r} \frac{2 \lambda}{\tau_0 \cosh (\omega r) + \frac{\tau_1}{\omega} \sinh (\omega r)} d r \\
        &\le \int_0^\infty e^{- 2 \omega r} \frac{4 \lambda}{\tau_0 (1 + e^{- 2 \omega r}) + \frac{\tau_1}{\omega} (1 - e^{- 2 \omega r})} d r \\
        &\le - \frac{2 \lambda}{\omega (\tau_0 - \frac{\tau_1}{\omega})} \biggl[\ln{\Bigl(\tau_0 (1 + e^{- 2 \omega r}) + \frac{\tau_1}{\omega} (1 - e^{- 2 \omega r}) \Bigr)}\biggr]_0^\infty \\
        &\le \frac{2 \lambda}{\omega (\frac{\tau_1}{\omega} - \tau_0)} \ln{\Bigl( 1 + \frac{\frac{\tau_1}{\omega} - \tau_0}{2 \tau_0} \Bigr)}.
\end{align*}
Then, we also have
\begin{align*}
    0 \le \int_0^\infty e^{- \omega r} \frac{1}{(\tau (r))^3} d r
        &\le \int_0^\infty e^{- \omega r} \frac{1}{\Bigl(\tau_0 \cosh (\omega r) + \frac{\tau_1}{\omega} \sinh (\omega r)\Bigr)^3} d r \\
        &\le \int_0^\infty e^{2 \omega r} \frac{16 \lambda}{\Bigl(\tau_0 (1 + e^{2 \omega r}) + \frac{\tau_1}{\omega} (e^{2 \omega r} - 1) \Bigr)^3} d r \\
        &\le - \frac{4 \lambda}{\omega (\tau_0 + \frac{\tau_1}{\omega})} \biggl[\Bigl(\tau_0 (1 + e^{2 \omega r}) + \frac{\tau_1}{\omega} (e^{2 \omega r} - 1) \Bigr)^{-2} \biggr]_0^\infty \\
        &\le \frac{2 \lambda}{\omega (\frac{\tau_1}{\omega} + \tau_0) \, \tau_0}.
\end{align*}
Then, as soon as $\tau_0$ is fixed, we get an expansion of $\mu_\infty$ with respect to $\tau_1$ with \eqref{eq:integral_gamma}:
\begin{equation}\label{eq:gamma-infini-app}
    \mu_\infty = \frac{\tau_1}{\omega} + \tau_0 + \O \( \frac{\ln{\tau_1}}{\tau_1} \).
\end{equation}

\subsection{Back to the PDE}

We now address the general case, in the sense that $u_0$ need not be
Gaussian. As announced in Proposition~\ref{prop:repulsive}, change the unknown $u$ to $v$, through the formula
\begin{equation}\label{eq:uv-repulsive}
  u(t,x)
  =\frac{1}{\tau_-(t)^{d/2}}v\left(t,\frac{x}{\tau_-(t)}\right)
\exp \Big({i\frac{\dot\tau_-(t)}{\tau_-(t)}\frac{|x|^2}{2}} \Big) \, ,
\end{equation}
with $\tau_-$ solution to \eqref{eq:tau-moins}, that is, 
\eqref{eq:tau0} \emph{where the last term is 
discarded} (this simplifies a little bit the computations, and the
discarded term brings no extra information any way). Then
\eqref{eq:logNLSquad} is equivalent, up to an irrelevant time
dependent phase (like previously),
to
\begin{equation}
  \label{eq:v}
  i\d_t v +\frac{1}{2\tau_-^2} \Delta v = \lambda |y|^2 v
  +\lambda v\ln |v|^2\, ,\quad v_{\mid t=0} =u_0 ,
\end{equation}
provided that we assume
\begin{equation}\label{eq:CItau}
  \tau_-(0)=1,\quad \dot \tau_-(0)=0.
\end{equation}
\subsubsection{Hamiltonian structure and consequences}
Set
\begin{align*}
  \mathcal E (t)& = \underbrace{\frac{1}{2\tau_-(t)^2}\|\nabla v(t)\|_{L^2}^2
  }_{=:\mathcal
    E_{\rm kin}(t)}+ 
  \lambda \int_{\R^d}|y|^2|v(t,y)|^2dy + \lambda
  \int_{\R^d}|v(t,y)|^2\ln |v(t,y)|^2dy .
\end{align*}
We compute
\begin{equation*}
  \dot  {\mathcal E} (t) =-2\frac{\dot\tau_-(t)}{\tau_-(t)} \mathcal  E_{\rm kin}(t).
\end{equation*}
We readily infer, with the same proof as in the case $\omega=0$ given
in \cite[Lemma~4.1]{CaGa18}, as already sketched in the proof of
Lemma~\ref{lem:apriori-v-partial}: 
\begin{lemma}\label{lem:apv}
 For $u_0\in \Sigma$ and $\lambda>0$, there holds
\begin{equation*}
 \sup_{t\ge 0}\left(\int_{{\mathbb R}^d}\left(1+|y|^2+\left|\ln
    |v(t,y)|^2\right|\right)|v(t,y)|^2dy +\frac{1}{\tau_-(t)^2}\|\nabla
  v(t)\|^2_{L^2({\mathbb R}^d)}\right)<\infty, 
\end{equation*} 
and
\begin{equation*}
 \int_0^\infty \frac {\dot \tau_-(t)}{\tau_-^3(t)}\|\nabla
  v(t)\|^2_{L^2({\mathbb R}^d)} dt<\infty.
\end{equation*}
\end{lemma}
\begin{proof}[Sketch of the proof]
Write $\mathcal E=\mathcal E_+-\mathcal E_-$, where
  \begin{equation*}
    \mathcal E_+ = \frac{1}{2\tau_-(t)^2}\|\nabla v(t)\|_{L^2}^2
 + 
  \lambda \int_{\R^d}|y|^2|v(t,y)|^2dy+ \lambda
  \int_{|v|>1}|v(t,y)|^2\ln |v(t,y)|^2dy.
\end{equation*}
Note that $\mathcal E_+$ is the sum of non-negative terms only.
Since $\mathcal E$ is non-increasing, for $t\ge 0$,
\begin{align*}
  \mathcal E_+(t) &\le \mathcal E(0)+\mathcal E_-(t)=  \mathcal E(0)+
  \lambda \int_{|v|\le 1}|v(t,y)|^2\ln \frac{1}{|v(t,y)|^2}dy\\
  &\lesssim 1+ \int_{\R^d} |v(t,y)|^{2-\eps}dy,
\end{align*}
for any $\eps>0$ sufficiently small. Then $\|v\|_{L^{2-\eps}}\lesssim
\|v\|_{L^2}^{1-\delta(\eps)}\|yv\|_{L^2}^{\delta(\eps)}$ with
$\delta(\eps)\to 0 $ as $\eps\to 0$, and we conclude like in the proof
of Lemma~\ref{lem:apriori-v-partial}.
\end{proof}
Since the $L^2$-norm of $v$ is independent of time, the boundedness of the momentum of $v$ is an evidence that $v$ is not
dispersive, while the boundedness of $|v|^2$ in LlogL is an evidence
that $v$ does not grow to infinity. Therefore $\tau_-$ describes the
dispersive rate of any solution to \eqref{eq:repulsive}.
\subsubsection{Center of mass}
Introduce
\begin{equation*}
  I_1(t) := \IM \int_{{\mathbb R}^d} \overline v (t,y)\nabla v(t,y)dy \, ,\quad I_2(t) :=
  \int_{{\mathbb R}^d} y|v(t,y)|^2dy \, . 
\end{equation*}
We compute:
\begin{equation}\label{eq:evolI1I2}
  \dot I_1= -2\lambda I_2 \,,\qquad \dot I_2 = \frac{1}{\tau_-^2(t)}I_1 \, .
\end{equation}
Set $\tilde I_2:=\tau_- I_2$. We compute, in view of \eqref{eq:tau0},
\begin{equation*}
  \frac{d^2}{dt^2} {\tilde I_2} = \omega^2\tilde I_2,
\quad
\text{hence}
\quad
  \tau_-(t) I_2(t) = a_0 \cosh(\omega t)+b_0\frac{\sinh(\omega t)}{\omega}.
\end{equation*}
In view of \eqref{eq:CItau},
\begin{equation*}
  a_0=I_2(0),\quad b_0 = \dot I_2(0)= I_1(0). 
\end{equation*}
We infer
\begin{equation*}
  I_2(t) \Eq t \infty \frac{a_0+b_0}{2} \frac{e^{\omega t}}{\tau_-(t)}\Eq
    t \infty \frac{a_0+b_0}{2\mu_\infty}=
    \frac{I_1(0)+I_2(0)}{2\mu_\infty}. 
  \end{equation*}
  Unlike in the case $\omega=0$ considered in \cite{CaGa18,FeAPDE},
  the asymptotic center of mass of $v$ 
  is not zero in general (while it is always zero in the context of
  Theorem~\ref{theo:logNLSdisp}); this is like in the scattering case, where
  the asymptotic profile has no reason to be centered at the origin
  (as shown by the existence of wave operators). 
  \smallbreak
  
  Integrating in time the first equation in \eqref{eq:evolI1I2}, we infer
  \begin{equation*}
    I_1(t)\Eq t \infty -\lambda t  \frac{I_1(0)+I_2(0)}{\mu_\infty},
  \end{equation*}
  provided that $I_1(0)+I_2(0)\not =0$. This suggests that $v$ is
  oscillatory. 

  \subsubsection{More on large time behavior }

  At this stage, we have established two differences with the dynamics
  of logNLS (without potential): the dispersion is the one dictated by
  the repulsive harmonic potential, that is, exponential, and in the
  dispersive frame (working with $v$), the asymptotic center of mass
  is not necessarily  zero.

  The next natural question would be to
  decide between a general asymptotic behavior for $|v|$ (like in the
  case $\omega=0$) or a complete scattering theory. We are not able to
  fully validate the second option, which is the most likely in view
  of the result on the center of mass, but the Gaussian case shows
  that no universal profile must be expected for the large time
  behavior of $|v|$.

  To see this, consider $d=1$ and two Gaussian initial data
  \begin{equation*}
    u_{01}(x)= e^{-x^2/2 + i\beta_1 x^2/2},\quad u_{02}(x)= e^{-x^2/2 +
      i\beta_2 x^2/2},\quad \beta_j>0.
  \end{equation*}
In other words, we start from two distinct Gaussian data, whose moduli
(hence all Lebesgue norms and momenta, for instance)
are equal. The corresponding solutions $u_1$ and $u_2$ are given by
the formula presented in Section~\ref{sec:gaussian}, boiling the
description down to the analysis of the ODE \eqref{eq:tau0}:
\begin{equation*}
  u_j(t,x) =
  b_j(t)e^{-x^2/(2\tau_j(t)^2)+i\dot\tau_j(t)x^2/(2\tau_j(t))},\quad
  \tau_j(t)\Eq t \infty \mu_{\infty,j}e^{\omega t}.
\end{equation*}
In view of
Section~\ref{sec:dispersion}, we have
\begin{equation*}
  \mu_{\infty,j} =
  1+\frac{\beta_j}{\omega}+\O\(\frac{\ln\beta_j}{\beta_j}\)\text{ as
  }\beta_j\to \infty,
\end{equation*}
and so for $\beta_2\gg \beta_1\gg 1$, the corresponding functions
$v_j$
have different (asymptotic centers of mass and)  asymptotic
profiles. The fact that $|v_j|^2$ converges strongly in $L^1$ to the
corresponding limiting Gaussian is straightforward; see
e.g. \cite[Corollary~1.11]{CaGa18}.

\bibliographystyle{abbrv}
\bibliography{biblio}
\end{document}